\UseRawInputEncoding

\documentclass[12pt]{amsart}
\usepackage{epsfig,color, amsmath}
\usepackage{esint}
\usepackage{hyperref}

  \hypersetup{colorlinks}
\hypersetup{citecolor=blue}
\hypersetup{urlcolor=blue}

\makeatother

\theoremstyle{definition}

\def\fnum{equation} 
\newtheorem{Thm}[\fnum]{Theorem}

\newtheorem{Def}[\fnum]{Definition}

\newtheorem{Rem}[\fnum]{Remark}
\newtheorem{Pro}[\fnum]{Proposition}

\numberwithin{equation}{section}

\newcommand{\Lip}{{\text {Lip}}}

  \newcommand{\R}{\ensuremath{\mathbb{R}}}

 \newcommand{\ba}{\begin{align*}}
 \newcommand{\ea}{\end{align*}}
 \newcommand{\na}{\nabla}
\newcommand{\la}{\langle}
\newcommand{\ra}{\rangle}

\newcommand{\p}{\partial}

\newcommand{\ep}{\epsilon}

\newcommand{\Hess}{{\text {Hess}}}

\newcommand{\dv}{{\text {div}}}

\newcommand{\eqr}[1]{(\ref{#1})}

\DeclareMathOperator{\RCD}{RCD}

\newcommand{\di}{\mathop{}\!\mathrm{d}}
\newcommand{\meas}{\mathfrak{m}}
\newcommand\inner[2]{\langle #1, #2 \rangle}

\title{unique continuation problem on RCD Spaces. I}
\author{Qin Deng}%
\author{Xinrui Zhao}%
\address{MIT, Dept. of Math.\\
77 Massachusetts Avenue, Cambridge, MA 02139-4307.}
\thanks{}

\email{}

\begin{document}

\maketitle

 \begin{abstract}
 In this note we establish the weak unique continuation theorem for caloric functions on compact $RCD(K,2)$ spaces and show that there exists an $RCD(K,4)$ space on which there exist non-trivial eigenfunctions of the Laplacian and non-stationary solutions of the heat equation which vanish up to infinite order at one point. We also establish frequency estimates for eigenfunctions and caloric functions on the metric horn. In particular, this gives a strong unique continuation type result on the metric horn for harmonic functions with a high rate of decay at the horn tip, where it is known that the standard strong unique continuation property fails.   
  \end{abstract}

\section{introduction}
 

In this note we study the weak and strong unique continuation properties for the heat equation on $\RCD(K,N)$ spaces. Consider the equation
 \begin{align}
     \Delta u-\partial_tu=0.
     \label{heat}
 \end{align}
Let $T > 0$. On an $\RCD(K,N)$ space $(X, \di, \meas)$, a solution to \eqref{heat} on $[0,T]$ is defined as follows. \begin{Def}
 A function $u: X\to \R$ is a solution (resp. sub-solution, super-solution) of \eqref{heat} if $u\in W^{1,2}_{loc} (X\times [0,T])$ and that, for all $\phi\in \Lip_0(X\times[0,T])$ with $\phi\geq 0$, we have \begin{align}
     \int_{X\times [0,T]}  
     -\inner{\na u}{\na \phi} -\phi\,\partial_tu\,\, d\meas dt=0\, (\text{resp. }\leq 0, \geq 0).
     \label{requ}
 \end{align}
 
\end{Def}

This paper is concerned with the validity of the strong and weak unique continuation property for the heat equation \eqref{heat}. To be precise, we say that 
\begin{itemize}
    \item The heat equation satisfies the weak unique continuation property on an $\RCD(K,N)$ space $(X, \di, \meas)$ if, for any solution $u \in W^{1,2}(X \times [0, T])$ of \eqref{heat}, if $u$ vanishes on any non-empty open subset of $X \times [0,T]$, then $u \equiv 0$.
    \item A function $u \in L^2(X \times [0,T])$ vanishes up to infinite order at some $(x_0, t_0) \in X \times [0,T]$ if there exists some $R > 0$ so that for any integer $N > 0$, there exists $C(N) > 0$ so that
    \begin{equation*}
        \int_{B_{r}(x_0) \times ((t_0 - r^2, t_0 + r^2) \cap [0, T])} |u|^2\, \, d\meas dt \leq Cr^N.
    \end{equation*}
    \item The heat equation satisfies the strong unique continuation property on an $\RCD(K,N)$ space $(X, \di, \meas)$ if, for any solution $u \in W^{1,2}(X \times [0,T])$ of \eqref{heat}, if $u$ vanishes up to infinite order at any $(x_0, t_0) \in X \times (0,T]$, then $u \equiv 0$.
\end{itemize}

Note that we restrict $t_0$ to be away from $0$ in the definition of strong unique continuation as it is always possible to solve the heat equation starting from some initial data $u(\cdot, 0)$, which vanishes up to infinite order (spatially) at some $x_0 \in X$. This would easily imply that the heat flow also vanishes up to infinite order at $(x_0,0)$ in the smooth case. We mention that unique continuation type results are related to giving an upper bound for the measure of the nodal set of non-trivial solutions and that, recently, there has also been work to establish lower bounds in the nonsmooth setting, see \cite{CF21, DPF21}.

Our first result in Section 1 gives the validity of the weak unique continuation property for compact $\RCD(K,2)$ spaces:
\begin{Thm}\label{WUC RCD(K,2)}
Let $(X, \di, \meas)$ be a compact $\RCD(K,2)$ space. The heat equation on $X$ satisfies the weak unique continuation property. 
\end{Thm}
As in \cite{DZ}, where the same theorem was shown for harmonic functions, the main idea to handle the spatial direction is to leverage the $C^{0}$-Riemannian structure of non-collapsed $\RCD(K,2)$ spaces from \cite{LS, OS}. To handle the time direction, we also show the time analyticity of solutions of \eqref{heat} following the recent \cite{Z}.

In \cite{DZ}, an $\RCD(K,4)$ space was constructed on which there exists non-trivial harmonic functions which vanish up to infinite order at some point. The heat flow of any such harmonic function would immediately give a counterexample to strong unique continuation for the heat equation as well. As such, we will be primarily interested in the strong unique continuation property for non-stationary solutions of the heat equation in this paper. We extend our result from \cite{DZ} as follows in Sections 3 and 5 respectively:
\begin{Thm}
    There exists an $\RCD(K,4)$ space and a non-trivial eigenfunction on it with eigenvalue $\mu\neq 0$ which vanishes up to infinite order at one point.
\end{Thm}
\begin{Thm}
    There exists an $\RCD(K,4)$ space and a non-stationary solution of the heat equation on it which vanishes up to infinite order at one point.
\end{Thm}
Besides these results, we also give frequency estimate on metric horns in Sections 2 and 4 and establish a strong unique continuation type result for eigenfunctions and caloric functions on metric horns, where the classical strong unique continuation property fails by the results of \cite{DZ}.

 We have to use the geometry setting of this equation to deal with the difficulty of coefficients which are non-Lipschitz. For discussions on smooth manifolds with Ricci curvarure lower bound and some related questions, see for example \cite{CM1,CC1,CM2,CC2,CC3,CM3,CN,CM4,CM5}. For references on the theory of $\RCD$ spaces, see \cite{S06a, S06b, LV09, AGS14, G15, AGMR15, AGS15, EKS15, DPG, MN19, BS20, CaM1, CaM2, CaMi, KM21, BNS22}. 
 
 \section*{Acknowledgement}
 We are very grateful to Prof.\,\,Tobias Colding for his interest on this unique continuation problem and his constant encouragements. Xinrui Zhao is supported by NSF Grant DMS 1812142 and NSF Grant DMS 1811267.
 
 \section{weak unique continuation of caloric functions on $\RCD(K,2)$ spaces}
 
 In this section we will consider the case of $\RCD(K,2)$ spaces. In the smooth setting, the Ricci curvature in $2$ dimensions completely determines sectional curvature and hence a Ricci lower bound translates to a sectional curvature lower bound. This is also true for any non-collapsed $\RCD(K,2)$ space $(X,\di, \meas)$.
 \begin{Thm}(\cite{LS})
If $\meas = c\mathcal{H}^2$ with $c \geq 0$ (i.e. $(X, \di, \meas)$ is non-collapsed), then $(X,\di)$ is an Alexandrov space with curvature bounded below by $K$.
 \label{thm:raequ}
 \end{Thm}
 
 Therefore, it suffices to consider $2$-dimensional Alexandrov spaces of curvature at least $K$ and collapsed $\RCD(K,2)$ spaces. Alexandrov spaces are known to have a generalized Riemannian structure by \cite{OS}. We refer to \cite{OS} for relevant definitions. 
 \begin{Thm}(\cite{OS})
 Let $(X,\di)$ be an $n$-dimensional Alexandrov space and denote $S_X$ as the set of singular points. Then there exists a $C^0$-Riemannian structure on $X\setminus S_X\subset X$ satisfying the following:
\begin{itemize}
    \item[(1)]There exists an $X_0\subset  X\setminus S_X$ such that $X\setminus  X_0$ is of $n$-dimensional Hausdorff measure zero and that the Riemannian structure is $C^{\frac{1}{2}}$-continuous on $X_0\subset X$;
    \item[(2)]The metric structure on $X\setminus S_X$ induced from the Riemannian structure coincides with the original metric of $X$.
\end{itemize}
 \label{Thm:c1}
 \end{Thm}

In a coordinate neighborhood $U$ given by the $C^0$-Riemannian structure of $X$, with corresponding metric $g_{ij}$, a solution $u$ of \eqref{heat} satisfies
\begin{align}
    \int_{U\times[0,T]}  -g^{ij}\,u_i\, \phi_j-\phi\,\partial_tu\,\, d\meas\,dt=0
\end{align}
for all $\phi\in Lip_0(X\times[0,T])$ supported in $U\times[0,T]$ (Cf. \cite{KMS01}). Note that since $g^{ij}$ may not be Lipschitz, the results of \cite{L} do not apply. We will instead use techniques which are special in the 2-dimensional case. For more details, see for example \cite{CH}. 

We proceed by taking local isothermal coordinates. To be precise, consider functions $\sigma,\rho:U \to \R^2$ satisfying
\begin{align}
    \sigma_x=\frac{g^{xy}\rho_x+g^{yy}\rho_y}{\sqrt{g^{xx}g^{yy}-(g^{xy})^2}}\\
    -\sigma_y=\frac{g^{xx}\rho_x+g^{xy}\rho_y}{\sqrt{g^{xx}g^{yy}-(g^{xy})^2}},
\end{align}
where derivatives are taken with respect to the coordinate chart for $U$. The existence of such functions is given by \cite{Mo}, see also \cite{BN,CH}. This is equivalent to solving the complex equation 
\begin{align}
     w_{\bar{z}}=\mu w_z,
 \end{align}
 where $\mu=\frac{g^{yy}-g^{xx}-2ig^{xy}}{g^{xx}+g^{yy}+2\sqrt{g^{xx}g^{yy}-(g^{xy})^2}}$. Note that 
 \begin{align}
     |w_z|^2=\frac{(\sigma_x\rho_y-\sigma_y\rho_x)}{4\sqrt{g^{xx}g^{yy}-(g^{xy})^2}}(g^{xx}+g^{yy}+2\sqrt{g^{xx}g^{yy}-(g^{xy})^2})
 \end{align}
 
We have
 \begin{align}
     u_x=u_\rho\rho_x+u_\sigma\sigma_x,\\
     u_y=u_\rho\rho_y+u_\sigma\sigma_y,
 \end{align}
 and so
 \begin{align}
    \notag\int_{U\times[0,T]}  &-(g^{xx}\,(u_\rho\rho_x+u_\sigma\sigma_x)\, (\phi_\rho\rho_x+\phi_\sigma\sigma_x))-(g^{xy}\,(u_\rho\rho_x+u_\sigma\sigma_x)\, (\phi_\rho\rho_y+\phi_\sigma\sigma_y))\\&-(g^{xy}\,(u_\rho\rho_y+u_\sigma\sigma_y)\, (\phi_\rho\rho_x+\phi_\sigma\sigma_x))-(g^{yy}\,(u_\rho\rho_y+u_\sigma\sigma_y)\, (\phi_\rho\rho_y+\phi_\sigma\sigma_y))-\phi\,\partial_tu\,\, d\meas\,dt=0.
\end{align}
 
 After rearranging the terms,
  \begin{align}
    \int_{U\times[0,T]}  \sqrt{g^{xx}g^{yy}-(g^{xy})^2}(\sigma_x\rho_y-\sigma_y\rho_x)(u_\rho\phi_\rho+u_\sigma\phi_\sigma)-\phi\,\partial_tu\,\, d\meas\,dt=0.
\end{align}
and so in the new coordinates, we have
 \begin{align}
    \int_{\R^2\times[0,T]}  (u_\rho\phi_\rho+u_\sigma\phi_\sigma)-\frac{\phi\,\partial_tu}{\sqrt{g^{xx}g^{yy}-(g^{xy})^2}(\sigma_x\rho_y-\sigma_y\rho_x)}\,\, dx\,dy\,dt=0.
\end{align}
 Thus in the coordinate $(\rho,\sigma)$ the equation becomes
 \begin{align}
     \Delta u-a\partial_t u=0,
 \end{align}
 where $a$ is measurable and H\"older on a full measure subset. Note that as now we have a possibly discontinuous coefficient in front of $\partial_t$, the result in \cite{L} cannot be used. We will instead use a geometrical argument to deal with this difficulty.
 
Now we are ready to state the main result of this section.
 \begin{Thm}
Let $u\in W^{1,2} (X\times [0,T])$ be a solution of \eqref{heat} on an $\RCD(K,2)$ space. If $u$ vanishes on a non-empty open set $\Gamma \subset X\times (0,T]$, then $u\equiv 0$ on $\text{proj}_X(\Gamma)\times [0,T]$. Moreover if $X$ is compact, then $u\equiv 0$ on $X\times [0,T]$.
 \end{Thm}
\begin{proof}

By a similar discussion as in \cite{DZ}, we can only consider the non-collapsed case. We assume that u vanishes on a non-empty open set $V\times (t_1,t_2)$. 
 

Since $u$ satisfies the heat equation, we have the following a priori estimates (cf.\cite[Remark 5.2.11]{G20}):
\begin{align}
    \|u(\cdot,t)\|_{L^2}\leq \|u(\cdot,0)\|_{L^2},\\
    \||\na u(\cdot,t)|\|_{L^2}\leq \frac{\|u(\cdot,0)\|_{L^2}}{\sqrt{t}},\\
   \|\partial_t u(\cdot,t)\|_{L^2} =\|\Delta u(\cdot,t)\|_{L^2}\leq \frac{\|u(\cdot,0)\|_{L^2}}{t}\label{time},
\end{align}
for any $t \in (0,T]$. 

We first prove that $u$ vanishes on $V\times [0,T]$ by showing that $u(x,\cdot)$ is analytic with respect to $t$. For the corresponding arguments on Riemannian manifolds see for example \cite{Z}. For the reader's convenience we sketch the argument here.
 
 As in \cite{LZ}, for any $(x_0,t_0)$ with $0 < t_0 < T$, $k>0$ and $j\leq k$, we consider \begin{align}
     H_j^1&=B(x_0,\frac{j\sqrt{t_0}}{\sqrt{2k}})\times [t_0-\frac{jt_0}{2k},t_0+\frac{jt_0}{2k}],\\ H_j^2&=B(x_0,\frac{(j+0.5)\sqrt{t_0}}{\sqrt{2k}})\times [t_0-\frac{(j+0.5)t_0}{2k},t_0+\frac{(j+0.5)t_0}{2k}].
 \end{align}
 So that $u$ is defined on these sets, we extend $u$ to be a solution of \eqref{heat} on $X \times [0, \infty)$.

By \cite{MN19}, see also \cite{CN}, we may choose a cutoff function $\phi_1(x,t)$ supported in $H^2_j$ such that $\phi_1=1$ on $H^1_j$, satisfying
\begin{align}
\label{cutoffestimates}
    |\na \phi_1(x,t)|^2+|\partial_t \phi_1(x,t)|\leq \frac{Ck}{t_0}.
\end{align}
Indeed, $\phi_1$ may be chosen to be a product of two cutoff functions on the spatial and time coordinates respectively, so that the spatial and time derivatives are well-defined without further measure theoretic arguments. 
It follows that
\begin{align}
    \int_{H^2_j}|\partial_tu|^2\phi_1^2\,d\meas\,dt&=-\int_{H_j^2}(\la\na\partial_t u,\na u\ra\phi_1^2+\partial_t u\la \na u,\na \phi_1^2\ra)\,d\meas\,dt\notag\\&=-\int_{H_j^2}(\frac{1}{2}\partial_t |\na u|^2\phi_1^2+\partial_t u\la \na u,\na \phi_1^2\ra)\,d\meas\,dt\notag\\
    &\leq \int_{H^2_j}(\frac{3}{4}|\partial_t u|^2\phi_1^2+\frac{Ck}{t_0}|\na u|^2)\,d\meas\,dt\notag,
\end{align}
where the last inequality follows from integration by parts, Young's inequality and \eqref{cutoffestimates}.
This gives that 
\begin{align}
    \int_{H_j^1}|\partial_tu|^2\,d\meas\,dt\leq \frac{Ck}{t_0}\int_{H_j^2}|\na u|^2\,d\meas\,dt\label{equ1}.
\end{align}
Similarly, we may choose a cutoff function $\phi_2(x,t)$ supported in $H^1_{j+1}$ such that $\phi_2=1$ on $H^2_j$, satisfying
\begin{align}
    |\na \phi_2(x,t)|^2+|\partial_t \phi_2(x,t)|\leq \frac{Ck}{t_0}.
\end{align}
Arguing as before, we have
\begin{align}
    \int_{H_{j+1}^1}|\na u|^2\phi_2^2\,d\meas\,dt&=-\int_{H_{j+1}^1}u\,\partial_t u\,\phi_2^2 +u\la\na u,\na \phi_2^2\ra\,d\meas\,dt\notag\\&=-\int_{H_{j+1}^1}\frac{1}{2}\,\partial_t u^2\,\phi_2^2 +u\la\na u,\na \phi_2^2\ra\,d\meas\,dt\notag\\
    &\leq \int_{H_{j+1}^1}(\frac{1}{4}|\na u|^2\phi^2_2+\frac{Ck}{t_0}u^2)\,d\meas\,dt,
\end{align}
and so
\begin{align}
    \int_{H_j^2}|\na u|^2\,dx\,dt\leq \frac{Ck}{t_0}\int_{H_{j+1}^1}u^2\,dx\,dt\label{equ2}.
\end{align}
By using \eqref{equ1},\eqref{equ2} inductively on $j$, we obtain
\begin{align}\label{kthtimederivativeestimate}
    \int_{H_1^1}(\partial_t^k u)^2\,dx\,dt\leq \frac{(Ck)^{2k}}{t_0^{2k}}\int_{H^1_{k+1}}u^2\,dx\,dt\leq \frac{(Ck)^{2k}}{t_0^{2k-1}},
\end{align}
where $C$ at the end also depends on the $L^2$-norm of $u_0$.

From \cite{HK,BB,R} we know that the $(p,p)$-Poincar\'e inequality holds on $\RCD(K,N)$ space for any $p \geq 1$ (see, for example, \cite[Section 4]{K15} for a careful discussion of Poincar\'e inequalties in the $\RCD(K,N)$ setting). Thus Moser iteration (\cite[Lemma 3.10]{HK}) shows that for $R<1$ there exist a mean value inequality \begin{align}
    |\partial_t^ku(x_0,t_0)|^2\leq C(\frac{R^{\nu_2}}{|B(x_0,R)|})\frac{1}{R^{\nu_2(1+\frac{2}{\nu_2})}}\int_{B(x,R)\times(t-R^2,t)}|\partial_t^ku|^2\,d\meas \,dt.
\end{align}
Thus by taking $R=\frac{\sqrt{t_0}}{\sqrt{2k}}$ and using \eqref{kthtimederivativeestimate}, we obtain
\begin{align}
    |\partial_t^ku(x_0,t_0)|^2\leq \frac{C}{|B(x_0,\frac{\sqrt{t_0}}{\sqrt{2k}})|(\frac{\sqrt{t_0}}{\sqrt{2k}})^2}\int_{H_1^1}|\partial_t^ku|^2\,d\meas\,dt\leq \frac{C}{|B(x_0,\frac{\sqrt{t_0}}{\sqrt{2k}})|}\frac{(Ck)^{2k+1}}{t_0^{2k}},\,\,t_0<k.
\end{align}
Using volume comparison, this implies that $u(x_0,\cdot)$ is analytic with respect to $t$. In particular, since $u$ vanishes on $V \times (t_1,t_2)$ by assumption, $u$ also vanishes on $V\times [0,T]$.

We now prove the second assertion of the theorem by contradiction. Assume $X$ is compact and $u\not\equiv 0$. 

Let $\phi_k$ be eigenfunctions of $-\Delta$ corresponding to eigenvalues $\lambda_k$ with $\|\phi_k\|_{L^2}=1$ and $0 = \lambda_0 < \lambda_1 \leq \lambda_2 \leq ... \to \infty$, see \cite{H18} for a discussion on this. It follows from the estimates obtained in the appendix of \cite{AHTP18} that $u$ admits the representation 
\begin{align}
    u(x,t)=\sum_{k=0}^\infty a_k e^{-\lambda_kt}\phi_k(x).
\end{align}


Let $j$ be the first index for which $a_j \neq 0$. As it does not affect the argument, we assume for simplicity that the dimension of the eigenspace corresponding to $\lambda_j$ is $1$. This gives that for any $x_0\in V$,
\begin{align}
    0=\mathop{lim}\limits_{t\to\infty}e^{\lambda_j t}u(x_0,t)=a_j \phi_j(x_0),
\end{align}
 where we have used \cite[Proposition 7.1]{AHTP18} to bound the values of $\phi_k(x_0)$ for $k > j$ to obtain the second equality. As we assumed that $a_j\neq 0$, we conclude that $\phi_j(x_0)=0$. This shows that $\phi_j(x_0)=0$ for all $k$ and $x_0\in V$. From \cite{DZ}, this implies that $\phi_j\equiv 0$, which is a contradiction. It follows that $u \equiv 0$. Note that similar arguments also work for Dirichlet problem on non-compact space.
 \end{proof}
 \begin{Rem}
     The previous argument actually shows the time analyticity of caloric functions with respect to time on any $\RCD(K,N)$ space, since that part of the argument does not require any assumptions on dimension. 
 \end{Rem}

 \begin{Rem}
 In \cite{L}, Fourier transform was used to reduce the problem to a solution of an elliptic equation on $X \times \R$. In this case, one does not have weak unique continuation for elliptic equations on $\RCD(K,3)$ spaces, so a different argument had to be used. 
 \end{Rem}

Finally, we recall a well-known counterexample given by Miller \cite{M} which indicates that in general we cannot expect weak unique continuation for parabolic operators with time-dependent coefficients even if the time-slices of the corresponding metric have a uniform Ricci curvature bound.

 \begin{Pro}(\cite{M})
 There exists a smooth function $u:\R^2\times [0,\infty)\to\R$  such that:\begin{align}
     u_t=((1+A(t)+a)u_x)_x+(bu_y)_x+(bu_x)_y+((1+C(t)+c)u_y)_y,
 \end{align}
 where $A=C=a=b=c=u=0$ on $t\geq T$, $a,b,c,u$ are smooth on $\R^2\times [0,\infty)$, $A(t),C(t)$ are smooth on $(0,T)$ and H\"older on $[0,\infty)$. Moreover, $u,a,b,c$ are period in x and y with period $2\pi$.
 \end{Pro}

 \section{Elliptic Frequency estimate on Metric horn}
In this section we will give a frequency estimate on the metric horn, which allows us to prove a form of unique continuation. Recall from \cite{DZ} that the standard formulation of strong unique continuation does not hold at the horn tip. The form of unique continuation we will prove in this section will therefore assume a higher order of decay at the horn tip, see Remark \ref{SUC on metric horn} at the end of the section for more details. 

 
On a weighted warped product $(X,dr^2+f^2(r)g_{S^{n-1}},e^{-\psi(r)}dvol)$, given function $\varphi$, we have that
\begin{align}
     \Delta \varphi=\partial_r^2\varphi+(n-1)\partial_r\varphi(log\,f)_r+ f^{-2}\Delta_{S^{n-1}}\varphi-\partial_r\varphi \partial_r\psi,
 \end{align}
 away from $r=0$.
 In the case of the standard metric horn, 
 \begin{align}
     f(r)&=\frac{1}{2}r^{1+\ep},\\ \psi(r)&=-(N-n)(1-\eta)log(r),
 \end{align}
 and so,
 \begin{align}
     \Delta \varphi=\partial_r^2\varphi+\partial_r\varphi(\frac{(n-1)(1+\ep)+(N-n)(1-\eta)}{r})+\frac{4}{r^{2+2\ep}}\Delta_{S^{n-1}}\varphi.
 \end{align}
 From the equation of Laplacian on metric horn, in particular, for $\varphi = r^\alpha$, 
 we have that \begin{align}\label{distancelaplacianestimate}
     \Delta \varphi =r^{\alpha-2}\alpha(\alpha+N-2+(n-1)\ep-(N-n)\eta) 
 \end{align}
  and\begin{align}
     \Hess\,\varphi=\alpha(\alpha-1)r^{\alpha-2}dr\otimes dr+(\frac{r^{1+\ep}}{2})^2\alpha(1+\ep)r^{\alpha-2}g_{S^{n-1}}.
 \end{align} 
 If we take $\alpha=2$, then
 \begin{align}
     \Hess\,\varphi= 2(dr^2+(1+\ep)(\frac{r^{1+\ep}}{2})^2g_{S^{n-1}}).
 \end{align}
 
Given an eigenfunction u with $\Delta u=\lambda u$, define scale-invariant quantities $I(r),E(r)$ and the frequency function $U(r)$ with respect to the level sets of the function $d_p$, where $p$ is the horn tip. For $r>0$, we denote
\begin{align}B_r=\{x\,|\,d_p(x)<r\},\\
\p B_r=\{x\,|\,d_p(x)=r\},\end{align} and define
\begin{align}
&I(r)=r^{1-n}\int_{\p B_r}u^2 d\meas_r\label{I},\\
& E(r)=r^{2-n}\int_{B_r}|\na u|^2+\lambda u^2 d\meas=r^{2-n}\int_{\p B_r}\frac{u}{|\na d_p|} \,\la\na u,\na d_p\ra d\meas_r, 
\label{E}\\
& U(r)=\frac{E(r)}{I(r)}\label{U},
\end{align}
where $\meas = e^\psi(r)dvol$ is the weighted volume measure of the metric horn and $\meas_r = e^\psi(r)dvol_r$ is the corresponding weighted area measure on $\p B_r$. 

We first compute the derivative of $I$. We remark that all  computations are done away from the cone tip $p$, so no regularity issues will arise from $p$ itself. Let $\phi$ be any smooth function compactly supported on $(0,\infty)$, we have
\begin{align*}
   & -\int I(t)\phi'(t)dt=-\int t^{1-n}\int_{\p B_t}u^2\phi'(t)dt=-\int u^2 d^{1-n}_p\phi'(d_p)\\&=-\int u^2 d^{1-n}_p\la\na d_p,\na\phi(d_p)\ra=\int \dv(u^2 d^{1-n}_p\na d_p)\phi(d_p)\\
   &=\int \frac{1}{2-n}(2u\la\na u, \na d_p^{2-n}\ra+u^2 \Delta d_p^{2-n})\phi(d_p)\\&=\int_t\int_{\p B_t} \frac{1}{(2-n)|\na d_p|}(2u\la\na u, \na d^{2-n}_p\ra+u^2 \Delta d^{2-n}_p)\phi(d_p)dt.
\end{align*}
This shows that 
\begin{align}\label{I'first}   I'(r)&=\int_{\p B_r} \frac{1}{(2-n)|\na d_p|}(2u\la\na u, \na d^{2-n}_p\ra+u^2 \Delta d^{2-n}_p)\\&=\frac{2E(r)}{r}
+\int_{\p B_r} \frac{1}{(2-n)|\na d_p|}u^2 \Delta d^{2-n}_p,\quad a.e.\,\, t\in(0,1),\notag\end{align}
and so by using \eqref{distancelaplacianestimate},
\begin{align}
    (\log I)'(r)-\frac{2U(r)}{r}=\frac{N-n+(n-1)\ep-(N-n)\eta}{r}.\label{logI}
\end{align}

For $E$, from the definition \eqr{E} and the coarea formula, we have that 
\begin{align}
    E'(r)=(2-n)\frac{E(r)}{r}+r^{2-n}\int_{\p B_r}({|\na u|^2}+\lambda u^2).\label{E'1}
\end{align}
Now consider the vector field $X=\frac{\la \na(d_p^2),\na u\ra\na u-\frac{1}{2}|\na u|^2\na(d_p^2)}{|\na d_p|^3}$, we have that 
\begin{align}
\dv(X)={\Hess(d_p^2)(\na u,\,\na u)-\frac{1}{2}|\na u|^2\Delta d_p^2}+\lambda u \la\na(d_p^2),\,\na u\ra
\label{div(X)}
\end{align}
By Integrating both sides in $B_r$ we have that 
\begin{align}  \label{shrink}
  \int_{\p B_r}\frac{2r}{|\na d_p|^3}|\la \na d_p,\,\na u\ra|^2-r\frac{|\na u|^2}{|\na d_p|}&=\int_{\p B_r}\la X,\,\na d_p\ra\\&=\int_{B_r}{\Hess(d_p^2)(\na u,\,\na u)-\frac{1}{2}|\na u|^2\Delta d_p^2}+\lambda u \la\na(d_p^2),\,\na u\ra.
  \notag
\end{align}
Combining \eqr{E'1} and \eqr{shrink} we have that 
\begin{align}
   E'(r)=&(2-n)\frac{E(r)}{r}+r^{2-n}\int_{\p B_r}\frac{2}{|\na d_p|^3}|\la \na d_p,\,\na u\ra|^2+\lambda u^2\\&-r^{1-n}\int_{B_r}({\Hess(d_p^2)(\na u,\,\na u)-\frac{1}{2}|\na u|^2\Delta d_p^2}+\lambda u \la\na(d_p^2),\,\na u\ra). \notag 
\end{align}

Now consider the frequency function $U(r)$. We have that 
\begin{align}\label{logU'}
    \frac{d}{dr}log\,U(r)=&\frac{E'(r)}{E(r)}-\frac{I'(r)}{I(r)}\\=&\frac{1}{E(r)}(-\frac{2E^2(r)}{rI(r)}+r^{2-n}\int_{\p B_r}|2\la \na d_p,\,\na u\ra|^2+\lambda u^2)-\frac{1}{I(r)}\int_{\p B_r} \frac{1}{2-n}u^2 \Delta d^{2-n}_p+\frac{2-n}{r}\notag\\&-\frac{1}{E(r)}(r^{1-n}\int_{B_r}\Hess(d_p^2)(\na u,\,\na u)-\frac{1}{2}|\na u|^2\Delta d_p^2+\lambda u \la\na(d_p^2),\,\na u\ra).\notag
\end{align}
 From Cauchy-Schwarz inequality we know that \begin{align}
     -\frac{2E^2(r)}{rI(r)}+2r^{2-n}\int_{\p B_r}|\la \na d_p,\,\na u\ra|^2\geq 0.\label{cs}
 \end{align}
 Combining \eqr{logU'} and \eqr{cs} we have that
 \begin{align}
    \frac{d}{dr}log\,U(r)\geq& \frac{\lambda r}{U(r)}-\frac{1}{I(r)}\int_{\p B_r} \frac{1}{2-n}u^2 \Delta d^{2-n}_p+\frac{2-n}{r}\label{U'1}\\&-\frac{1}{E(r)}(r^{1-n}\int_{B_r}\Hess(d_p^2)(\na u,\,\na u)-\frac{1}{2}|\na u|^2\Delta d_p^2-\frac{\lambda}{2} u^2 \Delta d_p^2)\notag.
\end{align}
 As we have that 
\begin{align}
    -\frac{1}{2-n}\int_{\p B_r} u^2 \Delta d^{2-n}_p=-\frac{1}{2}\int_{\p B_r} u^2 \dv( d^{-n}_p \na d_p^2)=-\frac{r^{-n}}{2}\int_{\p B_r} u^2 \Delta d_p^2+nr^{-n}\int_{\p B_r} u^2,
\end{align}
it follows that 
 \begin{align}\label{U'2}
 \begin{split}
    \frac{d}{dr}log\,U(r)&\geq \frac{\lambda r}{U(r)}+\frac{2}{r}-\frac{1}{E(r)}(r^{1-n}\int_{B_r}\Hess(d_p^2)(\na u,\,\na u))\\ &\geq \frac{\lambda r}{U(r)}-\frac{r^{1-n}\int_{B_r}2\ep r^{2+2\ep} g_{S^{n-1}}(\na u,\na u)}{r^{2-n}\int_{B_r}|\na u|^2}\\&\geq\frac{\lambda r}{U(r)}- \frac{2\ep}{r}.
\end{split}
\end{align}
 This shows that 
 \begin{align}
     (r^{2\ep} U(r))'\geq \lambda r^{1+2\ep},
 \end{align}
and so \begin{align}
      r^{2\ep}U(r)-U(1)\leq \frac{\lambda}{2+2\ep}(r^{2+2\ep}-1),
 \end{align}
 which is equivalent to 
 \begin{align}
     \label{Usecond}U(r)\leq  r^{-2\ep}U(1)+\frac{\lambda}{2+2\ep}(r^{2}-r^{-2\ep})\leq Cr^{-2\ep}.
 \end{align}
 Combining with \eqref{logI} we obtain
 \begin{align}
    (\log I)'(r)\leq Cr^{-1-2\ep}.
 \end{align}
 This gives \begin{align}
    \log I(r)\geq \log I(1)+C (1-r^{-2\ep}),
 \end{align}
and so in all \begin{align}
     I(r)\geq Ce^{-Cr^{-2\ep}}.
 \end{align}
 
 \begin{Rem}\label{SUC on metric horn}
 From \cite{DZ}, we know that strong unique continuation fails on metric horn. In particular, any harmonic function with $u(0)=0$ satisfies $u(x)=O(e^{-C(\log r)^2})$. However the discussion above shows that if an eigenfunction, in particular a harmonic function, vanishes at the tip up to order $u(x)=O(e^{-Cr^{-2\ep}})$, then $u\equiv 0$.
 \end{Rem} 
 
 \section{Failure of strong unique continuation of eigenfuctions on metric horn}
 In this section we will prove the following theorem:
\begin{Thm}
    There exists an $\RCD(K,4)$ space and a non-trivial eigenfunction on it with eigenvalue $\mu\neq 0$, which vanishes up to infinite order at one point.
\end{Thm}
 \begin{proof}Consider any modified metric horn $X$ constructed in \cite[Section 6]{DZ}, which uses techniques in \cite{WZ}. Let us denote the eigenfunctions on $S^{n-1}$ as $\{\varphi_i\}_{i=1}^\infty$ such that \begin{align}
     &\Delta_{S^{n-1}}\varphi_i=-\mu_i\varphi_i,\,\,\mu_i>0\\
     &\int_{S^{n-1}}\varphi_i^2dS=1.
 \end{align} 
 Assume $\varphi$ is an $L^2$ function which is smooth away from the tip, then $\varphi$ may be decomposed as 
 \begin{equation}\label{decomp}
    \varphi(r,\theta)=f_0(r)+\sum_{i=1}^\infty f_i(r)\varphi_i(\theta).
\end{equation}
Therefore, for any eigenfunction on metric horn with eigenvalue $-\mu$ we have, for the decomposition \eqref{decomp} and each $i$,
 \begin{align}
     f_i''(r)+f_i'(r)(\frac{(n-1)(1+\ep)+(N-n)(1-\eta)}{r})+\frac{4}{r^{2+2\ep}}f_i(r)(-\mu_i)+\mu f_i(r)=0,
 \end{align}
 for $r$ sufficiently close to $0$. We note that we do not have this formula for arbitrary $r$ since the modified metric horn was constructed using a gluing procedure. 
 
For the radial part, that is for $\mu_0=0$,
 \begin{align}
     f_0''(r)+f_0'(r)(\frac{(n-1)(1+\ep)+(N-n)(1-\eta)}{r})+\mu f_0(r)=0,
 \end{align}
 Denoting $c=(n-1)(1+\ep)+(N-n)(1-\eta)$, we have
 \begin{align}
     f_0(r) = k_1 r^{\frac{1 - c}{2}} J_{\frac{c - 1}{2}}(r \sqrt{\mu}) + k_2 r^{\frac{1 - c}{2}} Y_{\frac{c - 1}{2}}(r \sqrt{\mu}),
 \end{align}
 where $J_\nu,\,Y_\nu$ are Bessel functions. For a discussion of the properties of Bessel functions, see \cite[Appendix B]{M10}. We will use the fact that $J_{\frac{c-1}{2}}(r) \approx r^{\frac{c-1}{2}}$ and $Y_{\frac{c-1}{2}}(r) \approx r^{\frac{1-c}{2}}$ as $r \to 0$. Since $\varphi$ is assumed to be in $L^2$, we have that $f_0$ is in a weighted $L^2$ space where the weight for small $r$ is of the form $r^{c}$. Taking this into consideration with the asymptotics of $J$ and $Y$, we conclude that 
\begin{align}
    f_0(r)=Cr^{\frac{1 - c}{2}} J_{\frac{c - 1}{2}}(r \sqrt{\mu})\approx C\mu^{\frac{c-1}{4}}.
\end{align}

 For $\mu_i=i(n+i-2)$ and denoting $c=(n-1)(1+\ep)+(N-n)(1-\eta)$, we obtain
 \begin{align}
     f_i''(r)+f_i'(r)(\frac{c}{r})+\frac{4}{r^{2+2\ep}}f_i(r)(-\mu_i)+\mu f_i(r)=0.
 \end{align}
Define \begin{align}
     f_i(r)=g_i(r^{-\ep}).
 \end{align}
Then \begin{align*}
     f_i'(r)&=-\ep r^{-\ep-1}g_i'(r^{-\ep})，\\
      f_i''(r)&=\ep(\ep+1) r^{-\ep-2}g_i'(r^{-\ep})+\ep^2 r^{-2\ep-2}g_i''(r^{-\ep}).
 \end{align*}Thus \begin{align}
        g_i''(r)+\frac{\ep+1-c}{\ep} r^{-1}g_i'(r)+\frac{-4\mu_i}{\ep^2}g_i(r)+\frac{\mu}{\ep^2} r^{-\frac{2}{\ep}-2} g_i(r)=0.
 \end{align}
If we further define \begin{align}
    g_i(r)=k_i(r)r^{\frac{c-1-\ep}{2\ep}},
\end{align}
then\begin{align*}
    g_i'(r)&=k_i'(r)r^{\frac{c-1-\ep}{2\ep}}+k_i(r)\frac{c-1-\ep}{2\ep}r^{\frac{c-1-\ep}{2\ep}-1},\\
    g_i''(r)&=k_i''(r)r^{\frac{c-1-\ep}{2\ep}}+2k_i'(r)\frac{c-1-\ep}{2\ep}r^{\frac{c-1-\ep}{2\ep}-1}+k_i(r)\frac{c-1-\ep}{2\ep}(\frac{c-1-\ep}{2\ep}-1)r^{\frac{c-1-\ep}{2\ep}-2}
\end{align*} and so \begin{align}
     k_i''(r)=(\frac{c-1-\ep}{2\ep}(\frac{c-1-\ep}{2\ep}+1)r^{-2}+\frac{4\mu_i}{\ep^2}-\frac{\mu}{\ep^2} r^{-\frac{2}{\ep}-2})k_i(r).\label{eqnk}
 \end{align}
If we consider the solution with data $k_i^{(1)}(r_\mu)=1=(k_i^{(1)})'(r_\mu)$, where $r_\mu$ is chosen such that $1\geq\frac{c-1-\ep}{2\ep}(\frac{c-1-\ep}{2\ep}+1)r^{-2}-\frac{\mu}{\ep^2} r^{-\frac{2}{\ep}-2}\geq 0$ for $r\geq r_\mu$, that is $$r_\mu=\max\{\sqrt{\frac{c-1-\ep}{2\ep}(\frac{c-1-\ep}{2\ep}+1)},(\frac{c-1-\ep}{2\mu}(\frac{c-1-\ep}{2}+\ep))^{\frac{-\ep}{2}}\}.$$So for $r>r_\mu$ from $(\frac{4\mu_i}{\ep^2}+1)k^{(1)}_i(r)\geq(k^{(1)}_i)''(r)\geq \frac{4\mu_i}{\ep^2}k^{(1)}_i(r)$, we have \begin{align}e^{(\sqrt{\frac{4\mu_i}{\ep^2}}+1)(r-r_\mu)}\geq k_i^{(1)}(r)\geq \frac{1}{\sqrt{\frac{4\mu_i}{\ep^2}}}e^{\sqrt{\frac{4\mu_i}{\ep^2}}(r-r_\mu)}.\end{align} Then we have that \begin{align}
     k_i^{(2)}(r):=k_i^{(1)}(r)\int_{r}^\infty\frac{1}{(k_i^{(1)}(s))^2}ds.
 \end{align}
 is also a solution of \eqref{eqnk}. So we get\begin{align}
     \frac{1}{2({\frac{4\mu_i}{\ep^2}}+\sqrt{\frac{4\mu_i}{\ep^2}})}e^{(-\sqrt{\frac{4\mu_i}{\ep^2}}-2)(r-r_\mu)}\leq k_i^{(2)}(r)\leq \frac{\sqrt{\frac{4\mu_i}{\ep^2}}}{2}e^{(-\sqrt{\frac{4\mu_i}{\ep^2}}+1)(r-r_\mu)},\,\, \text{for}\, r>r_\mu .\label{eqna}
 \end{align}
 And also \begin{align}
     0<C(\ep,i)^{-1}\leq (k_2^{(2)})^2(r_\mu)+((k_2^{(2)})')^2(r_\mu)\leq C(\ep,i).
 \end{align}
From ODE theory we have that $k_i^{(1)},\,k_i^{(2)}$ form a basis of the solutions of \eqref{eqnk}. This tells us that\begin{align}
    f_i(r)=(ak_i^{(1)}+bk_i^{(2)})(r^{-\ep})r^{-\frac{c-1-\ep}{2}}.
\end{align} 
 As before, we have that $f_i$ is in a weighted $L^2$ space from assumption, this tells us that 
\begin{align}
    f_i(r)=bk_i^{(2)}(r^{-\ep})r^{-\frac{c-1-\ep}{2}}.
\end{align}  
From \eqref{eqna} we have the decay rate $f_i(r)=O(e^{-Cr^{-\ep}})$. This tells us that there exists eigenfunctions on $\RCD(K,4)$ spaces which are not zero but vanish up to infinite order at one point.\end{proof}\begin{Rem}
    In fact this argument tells us that any eigenfunction on this metric horn with zero integral on each link $\{r=r_0\}$ vanishes up to infinite order. 
\end{Rem} 

\begin{Rem}
Since \begin{align}
  \int_{r=0}^{r_\mu^{-\frac{1}{\ep}}}e^{2(-\sqrt{\frac{4\mu_i}{\ep^2}}-2)r^{-\ep}}r^{-(c-1-\ep)} 2^{1-n}r^c\,dr&=2^{1-n}\frac{(2(\sqrt{\frac{4\mu_i}{\ep^2}}+2))^{\frac{2(1+\ep)}{\ep}-1}}{\ep}\int_{2(\sqrt{\frac{4\mu_i}{\ep^2}}+2)r_\mu}^\infty e^{-r}r^{-\frac{2(1+\ep)}{\ep}}dr\notag\\&\geq c(n,\ep,\mu_i) e^{2(-\sqrt{\frac{4\mu_i}{\ep^2}}-2)r_\mu}r_\mu^{-\frac{2(1+\ep)}{\ep}},
\end{align}
the normalized coefficient $c_{\mu,i}$ for the eigenfunction $f_i(r)\varphi_i(\theta)$ with $L^2$-norm 1 corresponding to eigenvalue $\mu$  satisfies
\begin{align}
    c_{\mu,i}\leq c(n,\ep,\mu_i)e^{(\sqrt{\frac{4\mu_i}{\ep^2}}+2)r_\mu}r_\mu^{\frac{1+\ep}{\ep}}.
\end{align}
\end{Rem}

 \begin{Rem}
Let us use a more careful argument on $\varphi$. That is, as we know that \begin{align}
     |((k_i^{(2)})^2+((k_i^{(2)})')^2)|=|2(k_i^{(2)})'(k_i^{(2)}+(k_i^{(2)})'')|\leq 2|k_i^{(2)}(k_i^{(2)})'|C(\mu+1)\leq C(\mu+1)|(k_i^{(2)})^2+((k_i^{(2)})')^2|.
 \end{align}
 
 Thus if we denote the pasting radius as $\tilde{r}$, by the discussion we have \begin{align}
     C^{-1}e^{-C(\mu+1)(r_\mu-\tilde{r}^{-\frac{1}{\ep}})} \leq |((k_i^{(2)})^2+((k_i^{(2)})')^2)'|(\tilde{r}^{-\frac{1}{\ep}})\leq Ce^{C(\mu+1)(r_\mu-\tilde{r}^{-\frac{1}{\ep}})}.
 \end{align}
 
 As we know that when transform back \begin{align}
     C^{-1}e^{-C(\mu+1)(r_\mu-\tilde{r}^{-\frac{1}{\ep}})}\leq|f_i^2+(f_i')^2|(\tilde{r})\leq Ce^{C(\mu+1)(r_\mu-\tilde{r}^{-\frac{1}{\ep}})}.
 \end{align}
 
 As outside $\tilde{r}$ the space is a cone, the function is of the form $f(r) = k_1 r^{(1 - b)/2} J_{1/2 \sqrt{(b - 1)^2 + 4 a}}(\sqrt{c} r) + k_2 r^{(1 - b)/2} Y_{1/2 \sqrt{(b - 1)^2 + 4 a}}(\sqrt{c} r)$. From the asymptotics of Bessel functions we can see that if we consider the eigenfunction on $B_R$, the normalized coefficient satisfies 
 \begin{align}
    c_{\mu,i,R}\leq Ce^{C(\mu+1)(\mu^{\frac{\ep}{2}}-\tilde{r}^{-\frac{1}{\ep}})}R^{-\frac{n-1}{2}}.
\end{align}
 
 \end{Rem}
 
 \section{parabolic frequency estimate on metric horn}
 
 In this section we give a parabolic frequency estimate and the corresponding unique continuation type result on the metric horn. For simplicity, we consider the metric horn which is not modified. For previous discussions on parabolic frequency see for example \cite{Po,CM5}.
 
Consider the metric horn $(\R^n,\,\,g_{\ep,n},\,\,e^{(N-n)(1-\eta)log(r)}dvol)$ with $  g_{\ep,n}=dr^2+(\frac{1}{2}r^{1+\ep})^2g_{S^{n-1}}$. Note that we will use $X$ to denote the metric horn in the following. Let $p(x,y,t)$ be the heat kernel. Following \cite{Po}, for $x_0 \in X$ and $t_0 \geq 0$, we define the backward heat kernel $G_{x_0, t_0}: X \times (-\infty, t_0) \to \R$ as $G_{x_0,t_0}(x,t) := p(x,x_0,t_0-t)$. For simplicity of notation, we denote $G(x,t) = G_{o,0}(x,t)$. Let $u$ be a solution of \eqref{heat} on $X \times [-R_0^2, 0]$ for some $R_0 > 0$.  

 For $R_0 \geq R > 0$, define 
 \begin{align}
     I(R)&=R^2\int_{t=-R^2}|\na u|^2 G \, d\meas,\\
     D(R)&=\int_{t=-R^2}u^2G \, d\meas,\\
     N(R)&=\frac{I(R)}{D(R)}.
 \end{align}
 
From the definition of $G$, we have that\begin{align}
     \frac{\partial \log G}{\partial t}=-\Delta \log G-|\na \log G|^2,\end{align}
 and so
 \begin{align}
     \log G(x,t)=-\frac{(N+(n-1)\ep-(N-n)\eta)}{2}\log(-t)+\frac{|x|^2}{4t}.
 \end{align}
Therefore, 
\begin{align}
    I'(R)=&2R\int_{t=-R^2}|\na u|^2 Gd\meas-2R^3\int_{t=-R^2}(2\la\na u,\,\na u_t\ra G+|\na u|^2\partial_tG)d\meas\label{I'}\\=&2R\int_{t=-R^2}|\na u|^2 Gd\meas+4R^3\int_{t=-R^2}u_t^2 G+u_t\la\na u,\,\na G \ra d\meas \notag\\&-2R^3\int_{t=-R^2}|\na u|^2\partial_tGd\meas.\notag
\end{align}
Consider the vector field $X=\la\na G,\,\na u\ra\na u$, then
\begin{align}
     div(X)=\Hess_G(\na u,\,\na u)+\frac{1}{2}\la\na G,\,\na|\na u|^2\ra+\la\na G,\,\na u\ra\Delta u.
 \end{align}
 On the metric horn, 
 \begin{align}
     \frac{\Hess _G}{G}(x,t)=(\frac{|x|^2}{4t^2}+\frac{1}{2t})dr\otimes dr+\frac{(\frac{|x|^{1+\ep}}{2})^2}{2t}(1+\ep)g_{S^{n-1}},
 \end{align}
 and so
\begin{align}
    \int \Hess_G(\na u,\,\na u) +\la\na G,\,\na u\ra u_td\meas=\int\frac{1}{2}\Delta G|\na u|^2d\meas.
\end{align}
 So combining with \eqref{I'} we have that 
 \begin{align}
     I'(R)&=4R^3\int_{t=-R^2}(\frac{x}{2t}\cdot \na u+u_t)^2G(x,t)d\meas+4R^3\int_{t=-R^2}\frac{(\frac{|x|^{1+\ep}}{2})^2}{2t}\ep g_{S^{n-1}}(\na u,\,\na u)G(x,t)d\meas \notag\\&\geq 4R^3\int_{t=-R^2}(\frac{x}{2t}\cdot \na u+u_t)^2G(x,t)d\meas -\frac{2\ep}{R} I(R).
 \end{align}
Moreover,
 \begin{align}
     D'(R)=-2R\int_{t=-R^2}(2uu_tG+u^2\partial_t G)d\meas=-4R\int_{t=-R^2}u(u_t+\frac{x}{2t}\cdot \na u)Gd\meas,
 \end{align}
so from integration by parts we have that \begin{align}
     I(R)=-R^2\int_{t=-R^2}u(\partial_t u +\na u\cdot \frac{x}{2t}) G d\meas =\frac{R}{4}D'(R).
 \end{align}
From Cauchy-Schwarz inequality we obtain
 \begin{align}
     (\int_{t=-R^2}(\frac{x}{2t}\cdot \na u+u_t)^2G(x,t)d\meas)(\int_{t=-R^2}u^2G(x,t)d\meas)\geq (\int_{t=-R^2}u(\frac{x}{2t}\cdot \na u+u_t)G(x,t)d\meas)^2,
 \end{align}
 So we have that 
 \begin{align}
     (\log N)'(R)=\frac{I'}{I}(R)-\frac{D'}{D}(R)\geq -\frac{2\ep}{R},
 \end{align}
and so
\begin{align}
     \log N(R)\leq \log N(1)-2\ep\log R,
 \end{align}
 which is equivalent to that \begin{align}
     N(R)\leq CR^{-2\ep}.
 \end{align}
Since
 \begin{align}
     (\log D)'(R)=\frac{4I(R)}{RD(R)}=\frac{4}{R}N(R)\leq CR^{-1-2\ep},
 \end{align}
 we obtain \begin{align}
     \log D(R)\geq \log D(1)+C(1-R^{-2\ep}).
 \end{align}
 So finally
 \begin{align}
     D(R)\geq Ce^{-CR^{-2\ep}}.
 \end{align}
 This gives a strong unique continuation type result for solution of heat equation on the metric horn.
 
 \section{Failure of Strong unique continuation property of caloric on metric horn}
 In this section we will prove the following theorem:
 \begin{Thm}
    There exists an $\RCD(K,4)$ space and a non-stationary solution of heat equation on it which vanishes up to infinite order at one point.
\end{Thm}
 Similar as in the elliptic case, the heat equation becomes like
\begin{align}
     0=\Delta \varphi-\partial_t \varphi=\partial_r^2\varphi+\partial_r\varphi(\frac{(n-1)(1+\ep)+(N-n)(1-\eta)}{r})+\frac{4}{r^{2+2\ep}}\Delta_{S^{n-1}}\varphi-\partial_t \varphi.
 \end{align}
 Denote the eigenfunctions on $S^{n-1}$ as $\{\varphi_i\}_{i=1}^\infty$ such that \begin{align}
     &\Delta_{S^{n-1}}\varphi_i=-\mu_i\varphi_i,\,\,\mu_i>0\\
     &\int_{S^{n-1}}\varphi_i^2dS=1.
 \end{align}We can decompose it as $\varphi(r,\theta,t)=f_0(r,t)+\sum_{i=1}^\infty f_i(r,t)\varphi_i(\theta)$. So we have 
 \begin{align}
     \partial_r^2f_i(r,t)+\partial_rf_i(r,t)(\frac{(n-1)(1+\ep)+(N-n)(1-\eta)}{r})+\frac{4}{r^{2+2\ep}}f_i(r,t)(-\mu_i)-\partial_t f_i(r,t)=0
 \end{align}
 
 For $\mu_i=i(n+i-2)$ and denote $c=(n-1)(1+\ep)+(N-n)(1-\eta)$
 \begin{align}
       \partial_r^2f_i(r,t)+\partial_rf_i(r,t)(\frac{c}{r})+\frac{4}{r^{2+2\ep}}f_i(r,t)(-\mu_i)-\partial_t f_i(r,t)=0.
 \end{align}
 
 We first consider the solution of heat equation on the compact modified metric horn constructed in \cite[Section 6]{DZ}, where $p$ is the tip of the metric horn. We denote the eigenfunctions with eigenvalue $\nu_j$ with spherical components $\varphi_i$ as $q_{\nu_j}(r,\theta)=g_j(r)\varphi_i(\theta)$.
 
 Then we have that\begin{align}
     f_i(r,t)=\sum_j c_je^{-\nu_j t}g_j(r).
 \end{align} 
  Denote $$r_\mu=\max\{\sqrt{\frac{c-1-\ep}{2\ep}(\frac{c-1-\ep}{2\ep}+1)},(\frac{c-1-\ep}{2\mu}(\frac{c-1-\ep}{2}+\ep))^{\frac{-\ep}{2}}\}.$$
  
 By \cite{ZZ} we have $C_1j^{\frac{2}{N}}\leq\nu_j\leq C_2j^{2}$, thus $C_1j^{\frac{\ep}{N}}\leq r_{\nu_j}\leq C_2j^{\ep}.$
 
 From \eqref{eqna} we have that given any $100r_{\nu_{k}}\leq r^{-\ep} <100r_{\nu_{k+1}}$, which indicates that $\frac{C_1}{k+1}<r<C_2k^{-\frac{1}{N}}$, we have \begin{align}
     |f_i(r,t)|&\leq \sum_{j=1}^k|c_j|e^{-\nu_jt}c(n,\ep,\mu_i)e^{(\sqrt{\frac{4\mu_i}{\ep^2}}+2)r_{\nu_j}}r_{\nu_j}^{\frac{1+\ep}{\ep}}e^{(-\sqrt{\frac{4\mu_i}{\ep^2}}+1)r^{-\ep}}+\sum_{j=k+1}^\infty Ce^{-\nu_j t}\nu_j\notag\\
     &\leq Ck\cdot k^{1+\ep}e^{-cr^{-\ep}}+\sum_{j=k+1}^\infty Ce^{-Ctj^{\frac{2}{N}}}j^2\leq C r^{-(2+\ep)N}e^{-cr^{-\ep}}+Ce^{-Ctr^{-\frac{2}{N}}}r^{-2}\leq e^{-cr^{-\ep}}.
 \end{align}
 As the estimate is independent of k, we can see that $f_i$ vanishes up to infinite order at tip.
 
\begin{Rem}
    The argument above actually shows that all caloric functions on the modified metric horn which does not contain the radial part $f_0(r,t)$ vanishes up to infinite order at the tip.
\end{Rem}

\end{document}